\documentclass{article}
\usepackage{amssymb}
\usepackage{amsmath}
\usepackage{graphicx}
\usepackage[colorlinks]{hyperref}

\newcommand*{\LargerCdot}{\raisebox{-0.25ex}{\scalebox{1.5}{$\cdot$}}}
\newtheorem{theorem}{Theorem}

\newtheorem{corollary}[theorem]{Corollary}

\newtheorem{definition}[theorem]{Definition}
\newtheorem{example}[theorem]{Example}

\newtheorem{lemma}[theorem]{Lemma}

\newtheorem{proposition}[theorem]{Proposition}
\newtheorem{remark}[theorem]{Remark}

\newenvironment{proof}[1][Proof]{\noindent\textbf{#1.} }{\ \rule{0.5em}{0.5em}}
\begin{document}
\title{On the integrability of the isotropic almost complex structures and harmonic unit vector fields}

\author {Amir Baghban\\
Department of Mathematics, Faculty of Science\\
Azarbaijan Shahid Madani University\\
Tabriz 53751 71379, Iran\\
E-mail: amirbaghban87@gmail.com
\and 
Esmail Abedi\\
Department of Mathematics, Faculty of Science\\
Azarbaijan Shahid Madani University\\
Tabriz 53751 71379, Iran\\
E-mail: esabedi@azaruniv.ac.ir}

\maketitle


\renewcommand{\thefootnote}{}


\renewcommand{\thefootnote}{\arabic{footnote}}
\setcounter{footnote}{0}
\begin{abstract}
Aguilar introduced isotropic almost complex structures $J_{\delta , \sigma}$ on the tangent bundle of a Riemannian manifold $(M,g)$. In this  paper, some results will be obtained on the integrability of these structures. These structures with the Liouville $1$-form define a class of Riemannian metrics $g_{\delta , \sigma}$ on $TM$ which are a generalization of the Sasaki metric. Moreover, the notion of a harmonic unit vector field is introduced with respect  to these  metrics like as the Sasaki metric and the necessary and sufficient conditions for a unit vector field to be a harmonic unit vector field are obtained.

\end{abstract}

\section{Introduction}
In this paper, first the isotropic almost complex structures will be studied and then harmonic unit vector fields will be investigated.

Aguilar \cite{aguilar} introduced a class of almost complex structures on the tangent bundle of an arbitrary Riemannian manifold and proved that the necessary condition for integrability of such structures is that the base manifold has constant sectional curvature. Moreover, he proved the existence of special cases  when the base manifold is a space form. These special cases induce a class of $g$-natural metrics on $TM$. So it is natural to ask \textbf{is there any other integrable  structures on a space form?} The authors asked this question in the Mathoverflow and R. Bryant Proved there are many other cases. Bryant Characterized the  integrable   isotropic almost complex structures on $T\mathbb{R}^n$ and $TS^n$ and then answered to the stated problem. 

The authors presented  other equivalents to the integrability of isotropic almost complex structures on $T\mathbb{R}^n$ and $TS^n$ based on PDE's.

Using natural lifts of the Riemannian metric $g$ from the base manifold $M$ to the total space $TM$ of the tangent bundle, some new interesting geometric structures were studied (e.g. \cite{ dombrowski, S, PHR}). Maybe the best known Riemannian metric on the tangent bundle is the Sasaki metric introduced
by Sasaki in 1958 (see \cite{S}).  

Let $(M,g)$ be a compact Riemannian manifold and $(TM,g_s)$ be its tangent bundle equipped with the Sasaki metric. Moreover, suppose $(S(M),i^*g_s)$ is the unit tangent bundle of $(M,g)$ where $i:S(M) \longrightarrow TM$ is the inclusion map and $i^*g_s$ is the induced Sasaki metric on the unit tangent bundle. Denote by $\Gamma(TM)$ the set of all smooth vector fields on $M$. Moreover, let $\nabla$ be the Levi-Civita connection of $g$ and $ \Delta _g X$ be the rough Laplacian of vector field $X$ with respect to metric $g$.

Since, every vector field defines a map from $(M,g)$ to $(TM,g_s)$, it is natural to investigate the harmonicity of maps defined by vector fields.

Nouhaud \cite{dragomir} deduced that a vector field $X$ defines a harmonic map from $(M,g)$ to $(TM,g_s)$ if and only if $X$ is a parallel vector field. She found the expression of the Dirichlet energy associated to the vector field $X$ as
\begin{align*}
E(X)=\frac{n}{2}vol(M) + \frac{1}{2}\int _M ||\nabla X||^2 d\text{vol}(g),
\end{align*}
where $vol(M)$ is the volume of $M$ with respect to the metric $g$ and $||\nabla X||$ is the norm of $\nabla X$ as a $(1,1)$-tensor. She proved that parallel vector fields are the critical points of the Dirichlet energy defined from $C^{\infty}((M,g),(TM,g_s))$ to $\mathbb{R}^+$ by using the stated formula for $E(X)$. 

Ishihara \cite{dragomir} investigated that under certain conditions a map defined by a vector field from $(M,g)$ to $(TM,g_s)$ is a harmonic map. He calculated the tension field associated to a map from $(M,g)$ to $(TM,g_s)$ defined by a vector field and showed that the necessary and sufficient conditions for the harmonicity of this map is the vanishing of its Laplacian(We know that the conditions $\nabla X=0$ and $\Delta _g X=0$ are equivalent for an arbitrary vector field $X$ on a compact Riemannian manifold $(M, g)$).
 
 Medrano \cite{gil} investigated the critical points of the energy functional $E:\Gamma (TM) \to \mathbb{R}^+$ where $E$ is the restrected Dirichlet energy functional to the vector fields on a compact Riemannian manifold $(M,g)$. She proved that such vector fields are again parallel vector fields i.e., $\nabla X=0$.

Wood \cite{wood} introduced the notion of harmonic unit vector fields on a compact Riemannian manifold $(M,g)$ by restricting the Dirichlet energy functional to the unit vector fields and called the critical points of such functional, harmonic unit vector fields. Recall that when the Dirichlet energy functional is restricted to the unit vector fields, the vanishing of $ \nabla X$ ensures that the unit vector field $X$ is a harmonic unit vector field. Since, this is an strong condition for a unit vector field to be a harmonic unit vector field, it is natural to investigate the harmonicity of unit vector fields. Wood \cite{wood} demonstrated that a unit vector field $X$ is a harmonic unit vector field if and only if $\Delta _g X=||\nabla X||^2 X$.

The contributions on the harmonicity of maps defined by vector fields is not limited to the tangent bundles equipped with the Sasaki metric. Abbassi et al. \cite{abbassii} and  \cite{abbassi} studied the problem of determining of harmonicity of such maps with respect to the g-natural metrics. Dragomir and Perrone \cite{dragomir} introduced the problem of studying the harmonic unit vector fields when $TM$ equipped with the Riemannian metric $g_{\delta , \sigma}$ induced by an arbitrary isotropic almost complex structure $J_{\delta , \sigma }$. In this paper the harmonicity of a map defined by a vector field $X:(M,g)\longrightarrow (TM,g_{\delta ,\sigma})$ will be investigated.

The rest of the paper is organized as follows: Section 2 gives some preliminaries of tangent bundle and some notifications needed for integrability of  almost complex structures. In Section 3, some propositions on integrability of isotropic almost complex structures are resulted. It is notable that one of them is based on R. Bryants answer on Mathoverflow
\footnote{mathoverflow.net/questions/230574/solutions-of-equations-characterizing-a-complex-structure}
\footnote{mathoverflow.net/questions/234772/existence-of-non-constant-solutions-for-this-equations}
. In Section 4, we study the Riemannian metrics $g_{\delta , \sigma}$ on $TM$  and in particular, we calculate the Levi-Civita connection of these metrics. Section 5 is devoted to  achieve the necessary and sufficient conditions for a unit vector field to be a critical point of the restricted Dirichlet energy function to the unit vector fields.

\section{Preliminaries}

Assume $(M,g)$ is an n-dimensional Riemannian manifold and $\nabla$ represents the Levi-Civita connection of $g$. Moreover, let $\pi:TM \longrightarrow M$ be its tangent bundle and $K:TTM \longrightarrow TM$ be the connection map with respect to $\nabla$ where $\pi$ is the natural projection. The tangent bundle of $TM (TTM)$ can be split to vertical and horizontal vector sub-bundles $\mathcal{V}$ and $\mathcal{H}$, respectively, i.e., for every $v \in TM$, $T_vTM=\mathcal{V}_v  \oplus \mathcal{H}_v$. These sub-bundles have  the following properties
\begin{itemize}
\item $ \pi_{*v}\mid_{\mathcal{H} _v}: {\mathcal{H}}_v \longrightarrow T_{\pi(v)}M $ is an isomorphism,
\item $ K \mid_{{\mathcal{V}}_v }: \mathcal{V}_v \longrightarrow T_{\pi(v)} M $ is an isomorphism;
\end{itemize}
Where $\pi _{*v}$ is the differential map of $\pi$ at $v\in TM$.

Note that $X^v$ and $X^h$ , standing for the vertical and horizontal lifts of $X\in \Gamma (TM)$  , are the vector fields  on $TM$ defined by
\begin{itemize}
\item $X^v_u=(K \mid_{{\mathcal{V}}_u })^{-1}X({\pi (u)} )\in \mathcal{V}_u$'
\item $X^h_u=(\pi_{*v}\mid_{\mathcal{H} _u})^{-1}X({\pi (u)})  \in \mathcal{H}_u$;
\end{itemize}
$\forall \:$ $u \in TM$.

The Lie bracket of the horizontal and vertical vector fields at $u \in TM$ are expressed as follows
\begin{align}
[X^h,Y^h](u)&=[X,Y]^h_u-(R(X,Y)u)^v_u, \label{1}\\
[X^h,Y^v](u)&=(\nabla _X Y)^v_u,\label{2}\\
[X^v,Y^v](u)&=0(u),\label{3}
\end{align}
Where $0$ is the zero vector field on $TM$. Moreover, if we consider the vector field $X:M\longrightarrow TM$ as a map between manifolds, its derivative $X_*$ at a point $p$ in $M$ is given by
\begin{align}
X_{*p}(V)=V^h_{X(p)}+(\nabla _VX)^v_{X(p)} \hspace{1cm}\forall \: V\in \Gamma(TM).
\end{align}

Now, let $(M,J)$ be an almost complex manifold and $T^C(M)$ be the complexfication of $TM$. Define the spaces $T_x^{(0,1)}(M)$ and $T_x^{(1,0)}(M)$ of $T_x^C(M)$ as 
$$T_x^{(0,1)}(M)=\lbrace X_x+\sqrt{-1}J_xX_x|X(x)\in T_x(M)\rbrace,$$
and
$$   T_x^{(1,0)}(M)=\lbrace X_x-\sqrt{-1}J_xX_x|X(x)\in T_x(M)\rbrace . $$
It is a well known fact that $J$ is an integrable structure if and only if for all sections $A,B\in \Gamma (T^{(0,1)}(M))=\cup _{x\in M}T_x^{(0,1)}(M)$ we have $[A,B]\in \Gamma (T_x^{(0,1)}(M))$;  equivalently, for any two arbitrary members $\zeta _1 , \zeta _2$ of the dual space of $T^{(1,0)}(M)$\\$= \cup _{x\in M}T_x^{(1,0)}(M)$ (where denoted by $T^{(1,0)}(M)^*$) we have 
$$\zeta _1 \wedge \zeta _2 \in \bigwedge ^2  T^{(1,0)}(M)^*.$$
\begin{definition}
	Let $\eta , \eta ^1,...,\eta ^n$ be $1$-forms on a differentiable manifold $M$. We say that 
	$$d\eta \equiv 0 \qquad \text{mod} \:\lbrace \eta ^1,...,\eta ^n \rbrace,$$
	if and only if $$d\eta = \sum _{i,j}f_{ij}\eta ^i\wedge \eta ^j,$$ for some functions $f_{ij}$ on $M$.
\end{definition}
So, if $\zeta ^1,...,\zeta ^n$ be locally $(1,0)$-forms generating $\Gamma (T^{(1,0)}M^*)$, then $J$ is integrable if and only if 
$$ d\zeta \equiv 0  \qquad mod \lbrace \zeta ^1,...,\zeta ^n \rbrace  ,\qquad \forall \: \zeta \in \Gamma (T^{(1,0)}M^*) .$$

\section{Isotropic almost complex structures}
Isotropic almost complex structures are a generalized type of the natural almost complex structure $J_{1,0}:TTM\longrightarrow TTM $ given by
\begin{align*}
J_{1,0}(X^h)=X^v ,\hspace{1cm}J_{1,0}(X^v)=-X^h \hspace{1cm} \forall X \in \Gamma (TM).
\end{align*}
The isotropic almost complex structures determine an almost kahler metric
which kahler 2-form is the pullback of the canonical symplectic form on $T^*M$ to $TM$ via $\mathfrak{b} :TM \longrightarrow T^*M$. Moreover, Aquilar proved \cite{aguilar} that there is an isotropic complex structure on an open subset $\mathcal{A} \subset TM$ if and only if $\pi (\mathcal{A})$ is of constant sectional curvature.

\begin{definition}\cite{aguilar}\label{1000}
An almost complex structure $J$ on $TM$ is said to be isotropic with respect to the
Riemannian metric on $M$, if there are smooth functions $\alpha, \delta, \sigma :TM \longrightarrow \mathbb{R}$ such that $\alpha \delta - \sigma^2 =1$ and
\begin{align}
 JX^h=\alpha X^v + \sigma X^h, \hspace{0.5cm} JX^v=-\sigma X^v - \delta X^h \hspace{0.5cm} \forall X \in \Gamma (TM).\label{4}
\end{align}
\end{definition}
Hereafter, we will represent the isotropic almost complex structure associated to the maps $\alpha , \delta$ and $\sigma$ by $J_{\delta , \sigma}$.

 Suppose $(M,g)$ has constant sectional curvature $k$. Aquilar\cite{aguilar} proved that $J_{\delta , \sigma}$ is an integrable structure on an open subset $\mathcal{A} \subset TM$ if and only if the following equation holds 
\begin{align}\label{90000}
d\sigma +k \delta \Theta- \sqrt{-1}(1- \sqrt{-1}\sigma) \delta ^{-1}d\delta \equiv  0 \hspace{1cm}mod\hspace{2mm}\lbrace u^1,...,u^n\rbrace ,
\end{align}
where $\lbrace u^1,...,u^n  \rbrace $ are $1$-forms which generate the space of $(1,0)$-forms induced by $J_{\delta , \sigma}$ on $\mathcal{A} \subset TM$. When  $\alpha , \delta$ and $\sigma$ are functions of $E(u)=\frac{1}{2}g(u,u)$ then the above equation gives the following solutions for  $ \delta$ and $\sigma$ \cite{aguilar}
\begin{flalign}
&\delta ^{-1}=\sqrt{2kE+b}, \hspace{1cm}\sigma =0,& \label{90001}\\
&\delta ^{-2}=\frac{1}{2}\lbrace 2kE + b +\sqrt{(2kE+b)^2+4a^2k^2} \rbrace \hspace{1cm}\sigma =ak\delta ^2, a\neq 0,&\label{90008}
\end{flalign}
where $a,b \in \mathbb{R}$.

When $k=0$ we  prove that the equation (\ref{90000}) is equivalent to the $n$-complex equations stated in  the following Proposition.

\begin{proposition}\label{90013}
Let $(\mathbb{R}^n,\langle .,.\rangle )$ be the Euclidean space and $J_{\delta , \sigma}$ be an almost complex structure introduced as before. If we suppose $z=u+iv$ as a complex function on $\mathbb{R}^{2n}=T\mathbb{R}^n$ where $u,v$ are defined by $v=\frac{1}{\delta}$ and $u=\frac{\sigma}{\delta}$, respectively, then $J_{\delta , \sigma}$ is integrable if and only if 
\begin{align}
\frac{\partial z}{\partial x^l}+z\frac{\partial z}{\partial y^l}=0 \qquad \forall \:l, \:1\leq l\leq n,
\end{align}
where $(x^1,...,x^n)$ and $(x^1,...,x^n,y^1,...,y^n)$ are the natural coordinate systems for $\mathbb{R}^n$ and $\mathbb{R}^{2n}=T\mathbb{R}^n$, respectively.

\end{proposition}

\begin{proof}
	It is easy to check that the $1$-forms  $u^l=\sqrt{-1}\delta(dy^l-zdx^l)$, $1\leq l\leq n$ span the space of $(1,0)$-forms induced by $J_{\delta , \sigma}$. As the space spanned by $u^l$ is the same space spanned by $w^l=dy^l-zdx^l$, So $J_{\delta , \sigma}$ is integrable if and only if 
	\begin{align}
	dw^l\equiv 0 \hspace{1cm}mod\lbrace w^1,...,w^n\rbrace .\label{90010}
	\end{align}
	Since $dw^l=-dz\wedge dx^l$, the equation (\ref{90010}) can only happen if
	\begin{align}
	dz\equiv 0 \hspace{1cm}mod\lbrace w^1,...,w^n\rbrace .\label{90011}
	\end{align}
	But 
	\begin{align}
	dz&=\sum _{l=1}^n(\frac{\partial z}{\partial x^l}dx^l+\frac{\partial z}{\partial y^l}dy^l)\nonumber \\
	&\equiv \sum _{l=1}^n (\frac{\partial z}{\partial x^l}+z\frac{\partial z}{\partial y^l} )dx^l\hspace{0.5cm}mod\lbrace w^1,...,w^n\rbrace .
	\end{align}
	So the equation (\ref{90011}) happens if and only  if 
	\begin{align}
	\frac{\partial z}{\partial x^l}+z\frac{\partial z}{\partial y^l}=0, \qquad \forall \: 1\leq l\leq n,
	\end{align}
	and the proof is completed.
\end{proof}

It is natural to think of \textbf{is there any other integrable structure $J_{\delta , \sigma}$ except the types given by (\ref{90001}) and (\ref{90008})?} 

The following arguments are based on the Bryants answer in Mathoverflow. For more information we refer the reader to their web addresses mentioned in the introduction.

Let $\mathbb{R}^{n+1}$ be given its standard inner product (and extend it complex linearly to a
complex inner product on $\mathbb{C}^{n+1}$, which will be used below). Then 
\begin{align*}
S^n=\lbrace u \in \mathbb{R}^{n+1} |u.u=1  \rbrace
\end{align*}
and 
\begin{align*}
TS^n=\lbrace (u,v)\in \mathbb{R}^{n+1}|u.u=1 \hspace{0.5cm} and\hspace{0.5cm}u.v=0 \rbrace .
\end{align*}
Let $H_+=\lbrace x+iy|y  \gvertneqq 0 \rbrace  \subset \mathbb{C}$ be the upper-half line in $\mathbb{C}$. Define a mapping 

\begin{align*}
\Phi :TS^n \times H_+\longrightarrow \mathbb{C}^{n+1} \setminus  \mathbb{R}^{n+1}
\end{align*}
by $\Phi ((u,v),z)=v-zu$ where any vector $w=(w_1,...,w_{n+1}) \in \mathbb{R}^{n+1}$ is considered as a vector in  $\mathbb{C}^{n+1}$ like this vector $w=(w_1,0,...,w_{n+1},0)$ and so $zu, v-zu$ can be done naturally. $\Phi$ is a diffeomorphism and stablishes a foliation of $\mathbb{C}^{n+1}\setminus\mathbb{R}^{n+1}$ where the leaves of the foliation are the image of $\lbrace(u,v) \rbrace \times H_+$ for every $(u,v)\in TS^n$ under $\Phi$. 

The following Proposition  characterizes the integrable structures $J_{\delta ,\sigma}$ when the base manifold is an sphere.

\begin{proposition}
When $n \geq 2$, the almost complex structure $J_{\delta , \sigma}$ on an open subset $\mathcal{A}\subset TS^n$ is integrable if and only if the image of the mapping 
\begin{align*}
\Phi _z:\mathcal{A}\longrightarrow \mathbb{C}^{n+1}\setminus \mathbb{R}^{n+1},
\end{align*}
with the definition
\begin{align*}
\Phi _z(u,v)=\Phi ((u,v),z(u,v)),
\end{align*}
is a holomorphic hypersurface in $\mathbb{C}^{n+1}\setminus \mathbb{R}^{n+1}$. Where $z:\mathcal{A} \longrightarrow H_+$ is a mapping defined by $z(u,v)=\frac{\sigma +i }{  \delta}(u,v)$.
\end{proposition}
This Proposition implies that any holomorphic hypersurface of $\mathbb{C}^{n+1} \setminus \mathbb{R}^{n+1}$ that is transverse to the half-line foliation determined by $\Phi$ and intersects each such half-line in at most one point introduces a complex structure $J_{\delta ,\sigma}$. So,  one can construct a complex structure $J_{\delta , \sigma}$ on $TS^n$ by using certain holomorphic hypersurfaces of $\mathbb{C}^{n+1}\setminus \mathbb{R}^{n+1}$.

The following statements gives an other equivalent to the integrability of $J_{\delta ,\sigma}$ on an open subset  $TU\subset TS^n$ where $U$ is an open subset of the unit standard sphere $(S^n,g)$.

Let $E_0:U \subset S^n\to \mathbb{R}^{n+1}$ denote the (vector-valued) inclusion mapping. Let $E_1,...,E_n:   U \to\mathbb{R}^{n+1}$ be any (smooth) orthonormal tangential frame field extending $E_0$, i.e., $<E_a, E_b> = \delta_{ab}$ for $0\le a,b\le n$. Define functions $v_i:TU\to\mathbb{R}$ by $v_i(u,v) = E_i(u){\cdot}v$ for $1\le i\le n$, so that $v = \sum _{i=1}^nv_iE_i(u)$ for all $(u,v)\in TU$. 

One can consider $\zeta ^1,...,\zeta ^n$ as a basis for the $(1,0)$-forms on $TU$ with respect to $J_{\delta , \sigma}$. With the above notifications, the almost complex structure $J_{\delta , \sigma}$  on $\mathcal{A}$ is an integrable structure if and only if  the following equation holds
\begin{align}\label{q1}
d(z^2+v_1^2+...+v_n^2)\equiv 0 \hspace{0.5cm}mod \hspace{0.5cm} \lbrace \zeta ^1,... , \zeta ^n  \rbrace .
\end{align}

One can conclude the following proposition,
\begin{proposition}
Let $((x^1,...,x^n),U)$ be the conformally flat coordinae system on $U\subset (S^n,g)$, $(x^1,...,x^n,y^1,...,y^n)$ be the associated coordinate system on its tangent bundle and moreover let $J_{\delta , \sigma}$ be an isotropic almost complex structure on $TU$. Then $J_{\delta , \sigma}$ is an integrable structure if and only if 
\begin{align*}
\sum _{i=1}^n[\frac{\partial z}{\partial y^i}(y^{s_0}\mu _i-\mu _{s_0}y^i)-\frac{\partial z}{\partial y^{s_0}}y^i\mu _i]=y^{s_0}\lambda ^2-(\frac{\partial z}{\partial x^{s_0}}+z\frac{\partial z}{\partial y^{s_0}}),
\end{align*}
for all $s_0$ with $1\leq s_0 \leq n$.
\end{proposition}
\begin{proof}
	By considering $\lambda$ as the confrmal factor, the metric $g$ on $U$ can be written as follow
	\begin{align}
	g=\lambda ^2 (dx^1\otimes dx^1 +...+ dx^n \otimes dx^n).
	\end{align}
	So, one can define $E_i=\frac{1}{\lambda}\frac{\partial}{\partial x^i}$ for all $i=1,...,n$ and 
	\begin{align*}
	v_i(u,v)=\lambda (u)y^i, \qquad \forall \: (u,v)\in TU \qquad \text{and} \qquad i=1,...,n.
	\end{align*}
	It can be proved that in this coordinate system, $\zeta ^1 ,..., \zeta ^n$ are given by
	\begin{align*}
	\zeta ^k=dy^k+y^j(\delta _j^k\mu _ldx^l+\mu _jdx^k-\mu _kdx^j)-zdx^k,
	\end{align*}
	where $\mu _i=\frac{1}{\lambda}\frac{\partial \lambda}{\partial x^i}, \,  i=1,...,n$ and we used the Einstein summation in $l,j$. By using the all of above equations, one can get the follows
	\begin{align*}
	d_{(u,y)}(z^2+v_1^2+...+v_n^2)&=\sum _{i=1}^n[(2z\frac{\partial z}{\partial x^i}+2\mu _i ||y||^2)dx^i +2(y^i\lambda ^2 + z \frac{\partial z}{\partial y^i})dy^i],
	\end{align*}
	and
	\begin{align}\label{q2}
	&d_{(u,y)}(z^2+v_1^2+...+v_n^2)-2\sum _{i=1}^2(y^i\lambda ^2 + z \frac{\partial z}{\partial y^i})\zeta ^i (u,y)\nonumber \\
	&=2z\sum _{s=1}^n [\frac{\partial z}{\partial y^i}(y^s\mu _i-\mu _sy^i)+\frac{\partial z}{\partial y^s}(z-y^j\mu _j)-y^s\lambda ^2+\frac{\partial z}{\partial x^s}]dx^s.
	\end{align}
 Using the equations (\ref{q1}) and (\ref{q2}), one can get the conclusion.
\end{proof}

When  $k=0$ and  $\alpha , \delta$ and $\sigma$ are functions of $g(u,u)$, the Equations (\ref{90001}) and (\ref{90008}) show that they must be constant functions on $T\mathbb{R}^n$, but  there are non-constant examples $\delta ,\sigma$ such that define a complex structure $J_{\delta , \sigma}$.  One can characterize the integrable structures $J_{\delta , \sigma}$ on $T\mathbb{R}^n$ by using the holomorphic hypersurfaces of $\mathbb{C}^{n+1}\setminus \mathbb{R}^{n+1}$. 
\begin{example}
Let $v=\frac{1}{\delta}$ and $u=\frac{\sigma}{\delta}$ and define 
$$u(x,y)=\frac{x.y}{1+x.x},$$
and 
$$v=(x,y)=\frac{\sqrt{(1+x.x)(1+y.y)-(x.y)^2}}{1+x.x},$$
where $x=(x^1,...,x^n)$, $y=(y^1,...,y^n)$ and "." denotes the standard product on $\mathbb{R}^n$. It is easy to check that $z=u+iv$ satisfies the Proposition \ref{90013} and so $J_{\delta , \sigma}$ is a complex structure on a certain open subset of $T\mathbb{R}^n$.
\end{example}

\section{\texorpdfstring{A class of Riemannian metrics on $TM$} {}}

In this section, the Riemannian metric associated to the almost complex structure $J_{\delta ,\sigma}$ and the $1$-form Liouville $\Theta $  will be introduced and finally, thier Levi-Civita connection will be calculated.

 These metrics are a generalized type of the Sasaki metric and in some cases intersect g-natural metrics, specially the metrics induced by  those $J_{\delta , \sigma}$ whom Aguilar introduced by the equations (\ref{90001}) and (\ref{90008}).

Here after we suppose that $(M,g)$ is a Riemannian manifold and $\Theta \in \Omega ^1 (TM)$ is the  Liouvill 1-form defined by
\begin{align}
\Theta _v (A)=g_{\pi(v)}(\pi _*(A),v), \hspace{5mm}  A \in T_vTM, \hspace{5mm}  v \in TM.\label{5}
\end{align}
and $J_{\delta , \sigma}$ is an isotropic almost complex structure defined on the whole of $TM$.
\begin{definition} \cite{aguilar}\label{1001}
 Let $(M,g)$ be a Riemannian manifold and $J_{\delta,\sigma}$ be an isotropic almost complex structure on $TM$. Then the $(0,2)$-tensor $$g_{\delta,\sigma}(A,B)=d\Theta (J_{\delta,\sigma}A,B)$$
where $ A,B \in \Gamma (TTM)$,  defines a Riemannian metric on $TM$ if $ \alpha >0$ .
\end{definition}
Let $X,Y$ be local sections of $TM$. A simple calculation shows that
\begin{align}
g_{\delta , \sigma}(X^h,Y^h)=& \alpha g(X,Y) o\pi ,\label{6}\\
g_{\delta , \sigma}(X^h,Y^v)=& - \sigma g(X,Y) o\pi ,\label{7}\\ 
g_{\delta , \sigma}(X^v,Y^v)=& \delta g(X,Y) o\pi .\label{8}
\end{align}

\begin{remark}\cite{sarih}
Let $(M,g)$ be a Riemannian manifold and $G$ be a g-natural metric on $TM$. Then there are functions $\alpha _i,\beta _i :[0,\infty) \to \mathbb{R}$ for all $i=1,2,3$ such that for every $u, X, Y\in T_xM$, we have
\begin{align*}
&G_{(x,u)}(X^h,Y^h)=(\alpha _1 + \alpha _3)(r^2)g_x(X,Y)+(\beta _1 +\beta _3)(r^2)g_x(X,u)g_x(Y,u),\\
&G_{(x,u)}(X^h,Y^v)=\alpha _2(r^2)g_x(X,Y)+\beta _2(r^2)g_x(X,u)g_x(Y,u),\\
&G_{(x,u)}(X^v,Y^v)=\alpha _1(r^2)g_x(X,Y)+\beta _1(r^2)g_x(X,u)g_x(Y,u),
\end{align*}
where $r^2=g_x(u,u)$.
\end{remark}
 By letting 
$$ \alpha _1(g(u,u))=\delta (u),\,  \alpha _2(g(u,u))=-\sigma (u),\,  \alpha _3(g(u,u))=\alpha (u)-\delta (u),$$
(where $\sigma ,\, \delta$ are mappings defined in  (\ref{90001}) and (\ref{90008})  and $\alpha$ satisfies $\alpha \delta -\sigma ^2=1$) show that Aguilar in the Theorem $1$ of  \cite{aguilar} characterized the all of complex structures $J_{\delta , \sigma}$(in the existence part of the Theorem) which define g-natural metrics on the tangent bundle. 

Moreover it is worth mentioning that if the base manifold is the Euclidean space $\mathbb{R}^n$ then $\alpha ,\delta$ and $\sigma$ are constant functions if and only if the introduced metric on $TM$ by the complex structure $J_{\delta ,\sigma}$ is a g-natural metric. Moreover, in this case $\sigma$ must vanish and $\alpha =\frac{1}{\delta}$.
\begin{remark}
There are integrable structures $J_{\delta , \sigma}$ on the Euclidean space and Sphere such that the induced metrics by them are not g-natural metrics.
\end{remark}
\begin{remark}
It is  proved that the invariant isotropic complex structures $J_{\delta , \sigma}$ (where Aguilar called $J_{\delta , \sigma}$ to be invariant  if it is invariant by the natural action of the tangent maps of all the isometries of $M$)  on some open subset of $TM$ are those $J_{\delta , \sigma}$  with $\delta$ and $\sigma$ as in (\ref{90001}) and (\ref{90008}). More details are explained in the Proposition $3.1$ of \cite{aguilar}. 
\end{remark}

Next, we shall calculate the Levi-Civita connection of $g_{\delta ,\sigma}$. 
\begin{lemma}\label{1002}
\cite{sarih} Let $X$, $Y$ and $Z$ be any vector fields on $M$. Then
\begin{align}
&X^h(g(Y,Z)o\pi)=(Xg(Y,Z))o\pi ,\label{9}\\
&X^v(g(Y,Z)o\pi)=0.\label{10}
\end{align}
\end{lemma}

\begin{theorem} \label{1003}
Let $g_{\delta,\sigma}$ be a Riemannian metric on $TM$ as before. Then the Levi-Civita connection ̅$\bar{\nabla}$ of $g_{\delta,\sigma}$ at $(p,u) \in TM$ is given by

\begin{align}
\bar{\nabla}_{X^h} Y^h&=(\nabla _X Y)^h- \frac{\sigma}{\alpha} (R(u,X)Y)^h+ \frac{1}{2 \alpha} X^h (\alpha) Y^h +  \frac{1}{2 \alpha}Y^h(\alpha) X^h
\nonumber \\
& - \frac{\sigma}{\delta} (\nabla _X Y)^v-\frac{1}{2}(R(X,Y)u)^v- \frac{1}{2 \delta} X^h(\sigma) Y^v
\nonumber \\
&-\frac{1}{2\delta}Y^h(\sigma)X^v-\frac{1}{2}g(X,Y)\bar{\nabla} \alpha ,\label{11}\\
\bar{\nabla}_{X^h} Y^v&=-\frac{\sigma}{\alpha} (\nabla _X Y)^h + \frac{\delta}{2 \alpha}(R(u,Y)X)^h-\frac{1}{2\alpha}X^h(\sigma) Y^h
\nonumber\\
&+\frac{1}{2\alpha}Y^v(\alpha) X^h+(\nabla _X Y)^v+\frac{1}{2\delta}X^h(\delta)Y^v- \frac{1}{2\delta}Y^v(\sigma)X^v
\nonumber\\
&+ \frac{1}{2}g(X,Y) \bar{\nabla} \sigma ,\label{12}\\
\end{align}
\begin{align}
\bar{\nabla}_{X^v} Y^h&=\frac{\delta}{2\alpha}(R(u,X)Y)^h+ \frac{1}{2\alpha}X^v(\alpha)Y^h-\frac{1}{2\alpha}Y^h(\sigma)X^h
\nonumber\\
&- \frac{1}{2\delta}X^v(\sigma)Y^v+\frac{1}{2\delta}Y^h(\delta)X^v+\frac{1}{2}g(X,Y)\bar{\nabla}\sigma ,\label{13}\\
\bar{\nabla}_{X^v} Y^v&=-\frac{1}{2\alpha}X^v(\sigma)Y^h- \frac{1}{2\alpha}Y^v(\sigma)X^h+ \frac{1}{2\delta}X^v(\delta)Y^v
 \\
&+\frac{1}{2\delta}Y^v(\delta)X^v-\frac{1}{2}g(X,Y)\bar{\nabla}\delta .\label{14}
\end{align}
\end{theorem}
\begin{proof} We just prove (\ref{11}), the remaining ones are similar. Using Koszul formula, we have
\begin{align*}
2g_{\delta,\sigma}(\bar{\nabla}_{X^h} Y^h,Z^h)&=X^hg_{\delta,\sigma}(Y^h,Z^h)+Y^hg_{\delta,\sigma}(X^h,Z^h)-Z^hg_{\delta,\sigma}(X^h,Y^h)
\nonumber \\
&+g_{\delta,\sigma}([X^h,Y^h],Z^h)+g_{\delta,\sigma}([Z^h,X^h],Y^h)\\
&-g_{\delta,\sigma}([Y^h,Z^h],X^h).
\nonumber \\
\end{align*}
Using relations (\ref{1}), (\ref{6}) and (\ref{9}) gives us
\begin{align*}
2g_{\delta,\sigma}(\bar{\nabla}_{X^h} Y^h,Z^h)&=X^h(\alpha)g(Y,Z)+ \alpha Xg(Y,Z)+Y^h(\alpha)g(X,Z)\\
&+\alpha Yg(X,Z)-Z^h(\alpha)g(X,Y)-\alpha Zg(X,Y)\\
&+\alpha g([X,Y],Z)+\sigma g(R(X,Y)u,Z)+ \alpha g([Z,X]Y)\\
&+\sigma g(R(Z,X)u,Y)-\alpha g([Y,Z],X)\\
&- \sigma g(R(Y,Z)u,X).
\end{align*}
Using the properties of the Levi-Civita connection of $g$, we can get
\begin{align*}
2g_{\delta,\sigma}(\bar{\nabla}_{X^h} Y^h,Z^h)&=g(X^h(\alpha)Y,Z)+g(Y^h(\alpha)X,Z)-Z^h(\alpha) g(X,Y)\\
&+2\alpha g(\nabla _XY,Z)+ \sigma g(R(X,Y)u,Z)\\
&+ \sigma g(R(Z,X)u,Y)-\sigma g(R(Y,Z)u,X).
\end{align*}
Taking into account (\ref{6}) and the Bianchi’s first identity, we have
\begin{align*}
2g_{\delta,\sigma}(\bar{\nabla}_{X^h} Y^h,Z^h)&=g_{\delta,\sigma}(\frac{1}{\alpha}X^h(\alpha)Y^h+\frac{1}{\alpha}Y^h(\alpha)X^h-g(X,Y)\bar{\nabla}\alpha\\
 &+2 (\nabla _X Y)^h-\frac{2\sigma}{\alpha}(R(u,X)Y)^h,Z^h),
\end{align*}
so the horizontal component of $\bar{\nabla}_{X^h} Y^h$ is
\begin{align*}
h(\bar{\nabla}_{X^h} Y^h)&=\frac{1}{2\alpha}X^h(\alpha)Y^h+\frac{1}{2\alpha}Y^h(\alpha)X^h-\frac{1}{2}g(X,Y)h(\bar{\nabla}\alpha) + (\nabla _X Y)^h\\
&-\frac{\sigma}{\alpha}(R(u,X)Y)^h,
\end{align*}
where $\bar{\nabla}\alpha=h(\bar{\nabla}\alpha)+v(\bar{\nabla}\alpha)$ is the splitting of the gradient vector field of $\alpha$ with respect to $g_{\delta,\sigma}$ to horizontal and vertical components, respectively. Similarly the vertical component of $\bar{\nabla}_{X^h} Y^h$ is
\begin{align*}
v(\bar{\nabla}_{X^h} Y^h)&=-\frac{1}{2\delta}X^h(\sigma)Y^v-\frac{1}{2\delta}Y^h(\sigma)X^v-\frac{1}{2}g(X,Y)v(\bar{\nabla}\alpha)\\
&-\frac{\sigma}{\delta}(\nabla _X Y)^v-\frac{1}{2}(R(X,Y)u)^v.
\end{align*}
Using the equation $(\bar{\nabla}_{X^h} Y^h)=h(\bar{\nabla}_{X^h} Y^h)+v(\bar{\nabla}_{X^h} Y^h)$, the proof will be completed.
\end{proof}

\section{Harmonic unit vector fields}
In this section after calculating the tension field of a map defined by a unit vector field $X:(M,g) \to (TM,g_{\delta , \sigma})$, we shall compute its tension field as a map from $(M,g)$ to the $(S(M),i^*g_{\delta,0})$ and finally deduce some results on harmonic unit vector fields.
\subsection{\texorpdfstring{tension field of a map $X:(M,g)\longrightarrow (TM,g_{\delta,\sigma})$ defined by the vector field $X$} {}}

In this sub-section and next, the formula of the tension field associated to a map between Riemannian manifolds is retrieved from \cite{hajime}. So, one can refer to \cite{hajime} for more details.

Suppose $(M,g)$ and $(M',g')$ are two Riemannian manifolds, with $M$ compact. The Dirichlet energy associated to the Riamannian manifolds $(M,g)$ and $(M',g')$ is defined by
\begin{align*}
\begin{array}{ccc}
E:C^{\infty}(M,')	& \longrightarrow & \mathbb{R}^+ \\ 
f	& \longmapsto & \frac{1}{2}\int _M ||df||^2 d\text{vol}(g).
\end{array}
\end{align*}
 
Where $||df||$ is the Hilbert-Schmitd norm of $df$, i.e., $||df||^2=tr_g(f^*g')$ and $d\text{vol}(g)$ is the Riemannian volume form on $M$ with respect to $g$.
\begin{remark}\label{1013}
The expression $e(f)= \frac{1}{2}||df||^2=\frac{1}{2}tr_g(f^*g')$ is the so-called energy density of $f$.
\end{remark}
The critical points of $E$ are defined as harmonic maps. It is proved \cite{hajime} that a map $f:(M,g)\longrightarrow (M',g')$ is a harmonic map if and only if the tension field associated to $f$ vanishes identically. Therefore, one can investigate the harmonicity of a map defined by a vector field by calculating the tension field associated to this map. 

Suppose $(M,g)$ is a compact Riemannian manifold and $g_{\delta , \sigma}$ be a given metric on $TM$ defined as before and $W \in \Gamma (TM)$. Let $\lbrace V_1,...,V_n \rbrace$ be a local orthonormal basis for the vector fields on $M$, defined in a neighborhood of $p \in M$ such that $\nabla V_i=0$ at $p$. The Dirichlet energy of the map $W:(M,g) \longrightarrow (TM,g_{\delta , \sigma})$ defined by $W$ can be calculated as following
\begin{align*}
E(W)&=\frac{1}{2}\int _M ||d W||^2d\text{vol}(g)=\frac{1}{2}\int _M tr_gW^*(g_{\delta ,\sigma})d\text{vol}(g)
\nonumber \\
&=\frac{1}{2}\int _M\sum _{i=1}^n g_{\delta ,\sigma}(W_*(V_i),W_*(V_i))d\text{vol}(g)
\nonumber \\
&=\frac{1}{2}\int _M \sum _{i=1}^ng_{\delta ,\sigma}(V_i^h+ (\nabla _{V_i}W)^v,V_i^h+ (\nabla _{V_i}W)^v)d\text{vol}(g)
\nonumber \\
&=\frac{1}{2}\int _M (n\alpha -2 \sigma div(W)+ \delta ||\nabla W||^2) d\text{vol}(g).
\end{align*}

The tension field associated to the map $X:(M,g)\longrightarrow (TM,g_{\delta , \sigma})$ is locally defined by 
\begin{align*}
\tau _q (X) = \sum _{i=1}^n\lbrace \bar{\nabla}_{X_*(V_i)} X_*(V_i)- X_*(\nabla_{V_i}V_i )\rbrace (X(q)),
\end{align*}
for every $q$ in the domain of $V_i$, $i=1,...,n$. This definition is independent of the choice of $\lbrace V_1 ,...,V_n \rbrace $, so is a global definition on $M$. 

Therefore, one can write
\begin{align}\label{1017}
\tau _q (X)&=\sum _{i=1}^n\lbrace \bar{\nabla}_{X_*(V_i)} X_*(V_i)- X_*(\nabla_{V_i}V_i )\rbrace (X(q))
 \nonumber\\
&=\sum _{i=1}^n\lbrace \bar{\nabla}_{V_i^h + (\nabla _{V_i}X)^v}V_i^h + (\nabla _{V_i}X)^v
\nonumber \\
& -(\nabla _{V_i}V_i)^h-(\nabla _{\nabla _{V_i}V_i}X)^v\rbrace (X(q))
\nonumber\\
&=\sum _{i=1}^n\lbrace \bar{\nabla}_{V_i^h}V_i^h +\bar{\nabla}_{V_i^h}(\nabla _{V_i}X)^v +\bar{\nabla}_{(\nabla _{V_i}X)^v}V_i^h
\nonumber \\
&+\bar{\nabla}_{(\nabla _{V_i}X)^v}(\nabla _{V_i}X)^v -(\nabla _{V_i}V_i)^h-(\nabla _{\nabla _{V_i}V_i}X)^v\rbrace (X(q)).
\end{align}
\textbf{Notation:} Let $h$ be a $(0,2)$-tensor on  a Riemmanian manifold $(M,g)$ and $e_1,...,e_n$ be any orthonormal vectors at $p\in M$. We denote the $tr_gh=\sum _{i=1}^nh(e_i$\\$,e_i)$ by $tr_gh(\LargerCdot , \LargerCdot)$.

Using the equations (\ref{11}),...,(\ref{14}) and $(\nabla V_i)(p)=0$ in (\ref{1017}), we have the following formula of the tension field associated to $X$
\begin{align}\label{24}
\tau _p(X)&=\frac{1}{\alpha} \lbrace (1-\frac{n\alpha}{2})X_1- \frac{\alpha}{2}||\nabla X||^2 Y_1+\alpha div(X)Z_1+ tr_g(\nabla _{\LargerCdot})^v(\alpha)\LargerCdot
\nonumber \\
&-\sigma Ric(X)-\nabla _{\alpha Z_1-\sigma Z_2}X-tr_g(\nabla _{\LargerCdot}X)^v(\sigma)\nabla _{\LargerCdot}X
\nonumber \\
&-\sigma tr_g(\nabla _{\LargerCdot}\nabla _{\LargerCdot}X)+\delta tr_gR(X,\nabla _{\LargerCdot}X)\LargerCdot\rbrace ^h(X(p))
\nonumber \\
& + \frac{1}{\delta}\lbrace -\frac{n\delta}{2}X_2- \frac{\delta}{2}||\nabla X||^2 Y_2+\delta div(X)Z_2 -\alpha Z_1 +\sigma Z_2
\nonumber \\
& -tr_g(\nabla _{\LargerCdot}X)^v(\sigma)\LargerCdot+ \nabla _{\alpha Y_1 -\sigma Y_2}X + tr_g(\nabla_{\LargerCdot}X)^v(\delta)\nabla _{\LargerCdot}X
\nonumber \\
&+\delta \Delta _gX \rbrace ^v(X(p)),
\end{align}

Where $X_1=\pi _* ((\bar{\nabla}\alpha )o X)$, $X_2=K((\bar{\nabla}\alpha) o X)$, $Y_1=\pi _* ((\bar{\nabla}\delta ) oX)$, $Y_2=K((\bar{\nabla}\delta ) oX)$, $Z_1=\pi _* ((\bar{\nabla}\sigma ) oX)$ and $Z_2=K((\bar{\nabla}\sigma )oX)$ are vector fields on $M$ and $\Delta _gX=-\sum _{i=1}^n\lbrace  \nabla _{V_i}\nabla _{V_i}X-\nabla _{\nabla _{V_i}V_i}X\rbrace$ is the so-called rough Laplacian of $X$.
\begin{remark}
All of the traces stated in the equation (\ref{24}) are the traces of some tensors, so are independent of the choice of the orthonormal frame.
\end{remark}

\subsection{Tension field of unit vector fields}
The unit tangent bundle $S(M)$ on a Riemannian manifold $(M,g)$ is a fiber bundle on $(M,g)$ which its fibers at every point $p \in M$ is the set
\begin{align*}
S_p(M)= \lbrace v \in T_pM | \,g(v,v)=1 \rbrace .
\end{align*}
The tangent space of $S(M)$ is $\mathcal{H} \oplus \bar{\mathcal{V}}$, where $\mathcal{H} $ is the horizontal sub-bundle of $TTM$ with respect to the Levi-Civita connection $\nabla$ of $g$ and $\bar{\mathcal{V}}$ is the vector bundle on $S(M)$ such that at $(p,u) \in S(M)$ is defined by
\begin{align*}
\bar{\mathcal{V}}_{(p,u)}&= \lbrace Y_p^v \in \mathcal{V}_u |\, g(Y_p,u)=0, \hspace{0.1cm}\forall Y_p\in T_pM \rbrace 
\nonumber \\
&= \lbrace Y_p^v-g(Y_p,u)u^v |\,Y_p \in T_pM  \rbrace ,\label{27}
\end{align*}
where $Y_p^v \in \mathcal{V}_u $ is the vertical lift of $Y_p$ to  $\mathcal{V}_u$.
Let now $g_{\delta ,\sigma}$ be the Riemannian metric on $TM$ defined as before. Assum  $N^{g_{\delta , \sigma }}(p,u)= \sqrt{\alpha }(\frac{\sigma}{\alpha}u^h+u^v)$ is a vector field on $TM$, one can simply derive $N^{g_{\delta , \sigma }}(p,u)$ is normal unit vector field to $T_{(p,u)} S(M)$ with respect to the  $ g_{\delta ,\sigma}$.

We equip $S(M)$ with the induced metric $i^*g_{\delta , 0 }$, where $i:S(M) \longrightarrow TM$ is the inclusion map and represent its Levi-Civita connection by $\tilde{\nabla}$. We now deduce that the tension field of a unit vector field $X$ from $(M,g)$ to $ (S(M),i^* g_{\delta ,0 })$ is the tangent part of $\tau (X)$ with respect to the $g_{\delta , 0}$ i.e.,
\begin{align*}
 \tau _1(X)=  \mathrm{tan} \hspace{1mm} \tau (X),
 \end{align*}
 where $\tau (X)$ is the tension field associated to the map $X:(M,g)\to (TM,g_{\delta ,0})$.

Let $\lbrace V_1,...,V_n \rbrace$ be a local orthonormal basis for the vector fields, defined in a neighborhood of $p \in M$ which $\nabla V_i=0$ at $p\in M$. The tension field of a unit vector field $X$ is the tension field associated to the map $X:(M,g) \to (S(M),i^*g_{\delta ,0})$ defined as follow
\begin{align}
(\tau _1)_q (X) = \sum _{i=1}^n\lbrace \tilde{\nabla}_{X_*(V_i)} X_*(V_i)- X_*(\nabla_{V_i}V_i )\rbrace (X(q)),
\end{align}
for every $q$ in the domain of $V_i$, $i=1,...,n$. 

By using the Gauss formula for the Levi-Civita connections $\tilde{\nabla}$ of $i^*g_{\delta ,0}$ and $\bar{\nabla}$ of $g_{\delta , 0}$, one can get
\begin{align}
&(\tilde{\nabla}_{X_*V_i}X_*V_i)(X(q))\nonumber \\
&=\lbrace \bar{\nabla}_{X_*V_i}X_*V_i-\alpha g_{\delta ,0}(\bar{\nabla}_{X_*V_i}X_*V_i , X^v)X^v\rbrace (X(q)),
\end{align}
where,  $N^{g_{\delta ,0}}(X(q))=\alpha (X(q))X^v_{X(q)}$ is the normal unit vector field to $S(M)$ at $X(q) \in S(M)$ with respect to the metric $g_{\delta ,0}$. So by letting (32) in (31) we will get
\begin{align}
(\tau _1)_q(X)=\lbrace \tau (X)-\alpha \sum _{i=1}^ng_{\delta ,0}(\bar{\nabla}_{X_*V_i}X_*V_i , X^v)X^v\rbrace (X(q)).
\end{align}
Since $\nabla V_i=0$ at $p \in M$ for every $i=1,...,n$ and $g(\nabla _{\nabla _{V_i}V_i}X,X)=0$, one can write
\begin{align}
&(\tau _1)_p(X)= \lbrace \tau (X)\nonumber \\
&-\alpha \sum _{i=1}^n g_{\delta ,0}(\bar{\nabla}_{X_*V_i}X_*V_i  -(\nabla _{V_i}V_i)^h-(\nabla _{\nabla _{V_i}V_i}X)^v, X^v)X^v\rbrace (X(p))
\end{align}
and so we have the following formula for the $(\tau _1)_p(X)$
\begin{align}\label{1025}
(\tau _1)_p(X)=\tau _p(X)-g_{\delta ,0}(\tau _p(X) , N_{(p,X(p))})N_{(p,X(p))}.
\end{align} 
(\ref{1025}) shows that 
$$\tau _1(X)=  \mathrm{tan} \hspace{1mm} \tau (X)  $$
One can now calculate $(\tau _1)_p(X)$. Indeed, by substituting  $\sigma =0$ in (\ref{24}) gives us
\begin{align}
\tau _p (X) =&\lbrace  \frac{1}{\alpha}\lbrace tr_g ((\nabla _{\LargerCdot}X)^v(\alpha )\LargerCdot) +\frac{1}{\alpha}tr_g R(X,\nabla _{\LargerCdot}X){\LargerCdot}
\nonumber \\
&+(1-\frac{n\alpha}{2}+\frac{1}{2\alpha}||\nabla X||^2)X_1 \rbrace ^h (X(p)) \nonumber \\
& + \lbrace \alpha \lbrace \frac{1}{\alpha} \Delta _gX-\frac{1}{\alpha}\nabla _{X_1}X- \frac{1}{\alpha ^2}tr_g((\nabla _{\LargerCdot} X)^v(\alpha)\nabla _{\LargerCdot}X)
\nonumber \\
&+(\frac{1}{2\alpha ^3}||\nabla X||^2 - \frac{n}{2\alpha})X_2 \rbrace ^v \rbrace  (X(p)) .\label{30}
\end{align}
Taking into account (\ref{30}) and $N^{g_{\delta ,0}}(p,X(p))=\sqrt{\alpha}X^v(X(p))$ in (\ref{1025}) and using the definition of $g_{\delta ,0}$ gives us
\begin{align*}
(\tau _1)_p (X)&= \frac{1}{\alpha}\lbrace (1-\frac{n\alpha}{2}+\frac{1}{2\alpha}||\nabla X||^2)X_1 + tr_g((\nabla _{\LargerCdot}X)^v(\alpha){\LargerCdot}
\nonumber \\
&+\frac{1}{\alpha}tr_gR(X,\nabla _{\LargerCdot}X){\LargerCdot} \rbrace ^h  (X(p)) 
\nonumber \\
&+\alpha \lbrace \frac{-1}{\alpha} \nabla _{X_1}X -\frac{1}{\alpha ^2}tr_g((\nabla _{\LargerCdot}X)^v(\alpha)\nabla _{\LargerCdot}X) 
\nonumber \\
&+ \frac{1}{\alpha}\Delta _gX +(\frac{1}{2\alpha ^3}||\nabla X||^2 -\frac{n}{2\alpha} )X_2 -g(-\frac{1}{\alpha}\nabla _{X_1}X
\nonumber \\
&-\frac{1}{\alpha ^2}tr_g ((\nabla _{\LargerCdot}X)^v(\alpha)\nabla _{\LargerCdot}X)+ \frac{1}{\alpha}\Delta _gX
\nonumber \\
&+(\frac{1}{2\alpha ^3}||\nabla X||^2 -\frac{n}{2\alpha})X_2,X)X   \rbrace ^v (X(p)).
\end{align*}
Using the fact that $g(\nabla _{ V_i}X , X)=0$ for every $i=1,...,n$, we get the following expression for the $\tau _1 (X)$.
\begin{align}\label{1026}
(\tau _1)_p (X)&=\frac{1}{\alpha}\lbrace (1-\frac{n\alpha}{2}+\frac{1}{2\alpha}||\nabla X||^2)X_1 + tr_g((\nabla _{\LargerCdot}X)^v(\alpha){\LargerCdot}
\nonumber \\
&+\frac{1}{\alpha}tr_gR(X,\nabla _{\LargerCdot}X){\LargerCdot} \rbrace ^h (X(p)) 
\nonumber \\
&+\alpha \lbrace \frac{-1}{\alpha} \nabla _{X_1}X -\frac{1}{\alpha ^2}tr_g((\nabla _{\LargerCdot}X)^v(\alpha)\nabla _{\LargerCdot}X) 
\nonumber \\
&+ \frac{1}{\alpha}\Delta _gX +(\frac{1}{2\alpha ^3}||\nabla X||^2 -\frac{n}{2\alpha} )X_2-[\frac{1}{\alpha}g(\Delta _gX,X)
\nonumber \\
&+(\frac{1}{2\alpha ^3}||\nabla X||^2-\frac{n}{2\alpha})g(X_2,X)]X\rbrace ^v(X(p))
\end{align}

\begin{remark}\label{2002}
Since the condition $\tau _1 (X ) = 0$ has a tensorial character, as usual we can assume it as a definition of harmonic maps even when $M$ is not compact.
\end{remark}
The condition $\tau _1 (X)=0$ for the special vector fields can be reduced to a simple equation. Specially, for a parallel unit vector field $X$, we have the following corollary
\begin{corollary}\label{1070}
Let $(S(M),i^*g_{\delta ,0})$ be the unit tangent bundle equipped with the induced metric $i^*g_{\delta ,0}$ by the inclusion map $i:S(M) \longrightarrow TM$ for a compact Riemannian manifold $(M, g)$. Then, a map $X:(M,g)\longrightarrow (S(M),i^*g_{\delta ,0})$ defined by a parallel unit vector field $X$ on $M$ is a harmonic map if and only if
\begin{align}\label{1071}
(1-\frac{n\alpha}{2})X_1=0,
\end{align}
and
\begin{align}\label{1072}
X_2=||X_2||X.
\end{align}
\end{corollary}
\subsection{Variations through unit vector fields and harmonic unit vector fields}
In this section, we will give a  definition for harmonic unit vector fields analogous to a difinition for them with respect to the Sasaki and $g$-natural metrics. Then the necessary and suficient conditions for a unit vector field to be a harmonic unit vector field will be achieved.

 Let $(M,g)$ be compact Riemannian manifold and let
\begin{align*}
\begin{array}{ccc}
E:C^{\infty }((M,g), (S(M),i^*g_{\delta ,0})	& \longrightarrow & \mathbb{R}^+ \\ 
f	& \longmapsto & \frac{1}{2}\int _M ||df||^2 d\text{vol}(g),
\end{array}
\end{align*}
 
be the Dirichlet energy functional where $||df||^2=tr_g(f^*(i^*g_{\delta ,0}))$. It is a well-known fact in the harmonic theory that a map $f:(M,g) \longrightarrow (S(M),i^*g_{\delta ,0})$ is a critical point of $E$ if and only if the tension field $\tau _1$ associated to the map $f$ vanishes identically. According to the first variation formula, 
\begin{align} \label{1028}
\frac{d}{dt}|_{t=0}E(U_t)=-\int _M i^*g_{\delta ,0}(\mathcal{V},\tau _1(f))d\text{vol}(g).
\end{align}
 Where $U_t$ is a variation along $f$ for $|t| <\varepsilon$ which $U_0(x)=f(x)$ and $U_t(x) \in S(M)$ for every $x\in M$ and $\mathcal{V}(x)=\frac{d}{dt}|_{t=0} \lbrace t\longmapsto U_t (x)\rbrace$ is the variation vector field. 
 
 Now, let $X$ be a unit vector field on $M$ and $\mathcal{U}:M \times (-\varepsilon , \varepsilon) \longrightarrow S(M)$ be a smooth 1-parameter variation of $X$ through unit vecter fields i.e., $U_t \in \Gamma (S(M))$  for any $|t| < \varepsilon $ where $U_t(x)=\mathcal{U}(x,t)$, $x \in M$. It is proved \cite{dragomir} that the variation vector field $\mathcal{V}$ associated to this variation is of the form $\mathcal{V}(x)=V_{X(x)}^v$ where $V$ is a vector field on $M$ which is perpendicular to $X$ i.e., $g(X,V)=0$. One can now state the following proposition
 \begin{proposition}\label{1041}
Let $(S(M),i^*g_{\delta , 0})$ be the unit tangent bundle of a compact Riemannian manifold $(M, g)$ and $X$ be a unit vector field on $M$. Let $\mathcal{U}:M \times (-\varepsilon , \varepsilon ) \longrightarrow S(M)$ be a smooth 1-parameter variation of $X$ through unit vecter fields i.e., $U_t \in \Gamma (S(M))$  for any $|t| < \varepsilon $ where $U_t(x)=\mathcal{U}(x,t)$, $x \in M$. Then 
\begin{align}
\frac{d}{dt} \lbrace E(U_t) \rbrace \mid _{t=0}=- \int _M g_{\delta ,0} (V^v,\tau _1 (X)) \hspace{1mm} d\text{vol}(g),\label{32}
\end{align}
where $V^v$ is the variation vector field associated to the stated variation and $g(V,X)=0$.
\end{proposition}
\begin{proof}
Since, a variation through a unit vector field is a special variation among all of the variations, by using the first variation formula (\ref{1028}) and the fact that the variation vector field $\mathcal{V}$ is of the form $\mathcal{V}(x)=V_{X(x)}^v$ for every $x \in M$ and some vector field $V$ on $M$ which is perpendicular to $X$, the proposition will be proved.
\end{proof}
\begin{remark}\label{1040} \cite{dragomir}
Let $X$ be a unit vector field on $M$ and let 
\begin{align*}
\mathcal{S}=\lbrace  V\in \Gamma (TM)| g(V,X)=0 \rbrace .
\end{align*}
  Then for every $V \in \mathcal{S}$ there exists a smooth variation along $X$ through unit vector fields which its variation vector field is $V^v$. Indeed let $V$ be an arbitrary element of $\mathcal{S}$ and let us set 
\begin{align}
W_t=X+tV , \hspace{1cm}U_t=||W_t||^{-1}W_t ,\hspace{0.5cm}|t|<\epsilon .
\end{align}
It is not hard to check that $U_t$ is a variation along $X$ through unit vector fields with variation vector field $V^v$.
\end{remark}
The following definition is analogous to the definition of harmonic unit vector fields with respect to the Sasaki metric and $g$-natural metrics.
\begin{definition}\label{1029}
Let $(M,g)$ be a compact Riemannian manifold and $(S(M)$\\$,i^*g_{\delta ,0})$ be its unit tangent bundle equipped with the Riemannian metric $i^*g_{\delta ,0}$. A unit vector field $X \in \Gamma (S(M))$ is called a harmonic unit vector field if and only if the equation  (\ref{32}) vanishes for all vector field  $V \in \mathcal{S}$.
\end{definition}

The following theorem gives the necessary and sufficient conditions for a unit vector field to be a unit harmonic vector field.

\begin{theorem}\label{1050}
Let $(M,g)$ be a compact Riemannian manifold and $X$ be a unit vector field on $M$. Then, $X:(M,g) \longrightarrow (S(M), i^*g_{\delta ,0})$ is a harmonic unit vector field if and only if
\begin{align}
\Delta _gX&=[||\nabla X||^2+(\frac{1}{2\alpha ^2}||\nabla X||^2-\frac{n}{2})g(X_2,X)]X
\nonumber \\
&+(\frac{n}{2}-\frac{1}{2\alpha ^2}||\nabla X||^2)X_2+\nabla _{X_1}X
\nonumber \\
&+\frac{1}{\alpha}tr_g ((\nabla _.X)^v(\alpha)\nabla _.X),\label{33}
\end{align}
where $X_1$ and $X_2$ are the vector fields defined as before.
\end{theorem}

\begin{proof}
($\Longrightarrow $) Let $X$ be a harmonic vector field, we show that (\ref{33}) holds. Suppose $\tau _1(X)$ is the tension field associated to the map $X:(M,g) \longrightarrow (S(M),g_{\delta ,0})$ and let
\begin{align}
(\tau _1)_p(X)=\zeta X^h_{X(p)} + \lambda X^v_{X(p)} +V^v_{X(p)} + W^h_{X(p)}, \hspace{1cm}\forall \: p \in M, \label{34}
\end{align}
where $V$ and $W$ are perpendicular vector fields to $X$ and $\lambda ,  \zeta $ are smooth real functions on $X(M) \subseteq S(M)$. We show that $V=0$ and $\lambda =0$. From (\ref{34}), we have
\begin{align}
||V^v_{X(p)}||^2=g_{\delta,0}( (\tau _1)_p(X),V^v_{X(p)}), \qquad \forall \: p\in M.\label{35}
\end{align}
According to the Remark \ref{1040} and the Proposition \ref{1041} and the definition \ref{1029}, $ \int_M ||V^v||^2 d\text{vol}(g) = \int_M g_{\delta ,0}(\tau _1 (X) ,V^v) d\text{vol}(g)=0$. This shows that $V=0$, and the equation
\begin{align}
 (\tau _1)_p(X)=\tau _p(X) - \alpha(X(p))g_{\delta ,0}(\tau _p (X) , X^v_{X(p)}) X^v_{X(p)},
\end{align}
 shows that $\tau _1(X)$ hasn't any component in direction of $X^v$, i.e., $\lambda =0$. From (\ref{34}) and $V=0$ and $\lambda =0$, one can get $K(\tau _1 (X))= 0$. On the other hand, from (\ref{1026}) we have
\begin{align}
K(\tau _1 (X))&=\alpha \lbrace \frac{-1}{\alpha} \nabla _{X_1}X -\frac{1}{\alpha ^2}tr_g((\nabla _{\LargerCdot}X)^v(\alpha)\nabla _{\LargerCdot}X) 
\nonumber \\
&+ \frac{1}{\alpha}\Delta _gX +(\frac{1}{2\alpha ^3}||\nabla X||^2 -\frac{n}{2\alpha} )X_2-[\frac{1}{\alpha}g(\Delta _gX,X)
\nonumber \\
&+(\frac{1}{2\alpha ^3}||\nabla X||^2-\frac{n}{2\alpha})g(X_2,X)]X\rbrace.\label{36}
\end{align}
Using $g( \Delta _gX,X)=\frac{1}{2}\Delta (||X||^2) + ||\nabla X||^2=||\nabla X||^2$ and $K(\tau _1 (X))=0$, we get
\begin{align}
\Delta _gX&=[||\nabla X||^2+(\frac{1}{2\alpha ^2}||\nabla X||^2-\frac{n}{2})g(X_2,X)]X
\nonumber \\
&+(\frac{n}{2}-\frac{1}{2\alpha ^2}||\nabla X||^2)X_2+\nabla _{X_1}X
\nonumber \\
&+\frac{1}{\alpha}tr_g ((\nabla _{\LargerCdot}X)^v(\alpha)\nabla _{\LargerCdot}X),
\end{align}
($\Longleftarrow$) Let (\ref{33}) holds, we show that $X$ is a harmonic unit vector field. Substituting (\ref{33}) in (\ref{36}) gives us, $K(\tau _1 (X))=0$, i.e., the vertical part of $\tau _1(X)$ is zero. Proposition \ref{1041} with $K(\tau _1 (X))=0$ give us
\begin{align}
\frac{d}{dt} \lbrace E(\mathcal{U}_t) \rbrace \mid _{t=0}=- \int _M g_{\delta , 0} (V^v,\tau _1 (X)) \hspace{1mm} d\text{vol}(g)=0,\label{39}
\end{align}
where $\mathcal{U}_t$ is an arbitrary variation along $X$ through unit vector fields and $\mathcal{V}=V^v$ is its variation vector field. This shows that $X$ is a critical point of $E$ and the proof is completed. Note that the vertical and the horizontal sub-bundles are perpendicular to each other with respect to $g_{\delta , 0}$.
\end{proof}

Note that the theorem \ref{1050} shows that $X:(M,g)\to (S(M),i^*g_{\delta ,0})$ is a harmonic unit vector field if and only if the vertical part of $\tau _1(X)$ is zero.
\begin{remark}
Let $(M,g)$ be a compact Riemannian manifold and let $(S(M),i^*g_s)$ be its unit tangent bundle equipped with the induced Sasaki metric by inclusion map $i:S(M)\longrightarrow TM$. It is proved \cite{wood} that a unit vector field $X:(M,g)\longrightarrow (S(M),i^*g_s)$ is a harmonic vector field if and only if $\Delta _g X=||\nabla X||^2X$. 
\end{remark}
\begin{corollary}\label{t1}
If we suppose thet $(M,g)$ is a Riemannian manifold of constant sectional curvature $k$ and  $g_{\delta,0}$ is defined by $\alpha=\delta ^{-1}=\sqrt{2kE+b}$ and $\sigma =0$ (mappings defined in the equation (\ref{90001})) then  $X:(M,g) \to (S(M), i^*g_{\delta ,0})$ is a harmonic unit vector field if and only if
\begin{align}\label{t2}
\Delta _gX&=||\nabla X||^2X,
\end{align}
where $E(u)=\frac{1}{2}g(u,u)$.

\begin{proof}
Let $(x^1,...,x^n)$ be a locally coordinate system on $M$ such that the metric $g$ is of the form $g=\lambda ^2\sum _{i=1}^ndx^i \otimes dx^i$ with respect to this coordinate system, where $\lambda$ is the conformal factor. Let
$$X_i=\frac{1}{\lambda}\frac{\partial}{\partial x^i},\qquad  i=1,...,n,$$
be the locally orthonormal vector fields on $M$. Furthermore, suppose 
$$\theta ^i= \lambda dx^i\,, i=1,...,n,$$
be thier dual $1$-forms. Suppose  $(x^1,...,x^n,y^1,...,y^n)$ is the associated locally coordinate system on $TM$ and   $\xi ^i=\lambda dy^i\,,i=1,...,n$ are locally defined $1$-forms on $TM$. From \cite{aguilar}, we know that 
$$dE=\sum _{i=1}^n \theta ^i \xi ^i,$$
and so for the given $\alpha$ one can deduce that 
$$d\alpha =-\frac{ \lambda k}{\alpha} \sum _{i=1}^n \theta ^i  dy^i,$$
and this implies that the gradient vector field of $\alpha$ with respect to the given metric $g_{\delta ,0}$ is given by
\begin{align}\label{e1}
(\bar{\nabla}\alpha )V=-\frac{k}{\lambda (\pi oV)} V_V^v ,
\end{align}
for all vectors $V \in TM$. The equation (\ref{e1}) shows that the vector fields $X_1,X_2$ stated in the theorem \ref{1050} are $X_1=0$ and $X_2(p)=-\frac{k}{\lambda (p)}X(p)$ for all $p\in M$. Since $X$ is a unit vector field and $(\bar{\nabla}\alpha )X(p)=-\frac{k}{\lambda (p)}X_{X(p)}^v$ then 
$$(\nabla _ YX)^v(\alpha)=0,$$
for all $Y\in TM$. Now, using the stated properties the equation (\ref{33}) can be easily reduced to this equation:
$$\Delta _gX =||\nabla X||^2X.$$
\end{proof}
\end{corollary}
The following example shows that the Hopf vector fields on $S^3(1)$ are harmonic unit vector fields when $g_{\delta ,0}$ is the given metric in the corollary \ref{t1}.
\begin{example}
Let $(S^3,g)$ be the standard unit $3$-sphere and $g_{\delta , 0}$ be the introduced metric in the last corollary with $0\leq b$ and $k=1$. Moreover, let $J_1,J_2,J_3:T\mathbb{R}^4\to T\mathbb{R}^4$ be three complex structures on the 4-dimensional Euclidean space $(\mathbb{R}^4 , \langle .,. \rangle) $ defined by
\begin{align*}
&J_1(v_1,v_2,v_3,v_4)=(-v_2,v_1,-v_4,v_3),\\
&J_2(v_1,v_2,v_3,v_4)=(v_3,-v_4,-v_1,v_2),\\
&J_3(v_1,v_2,v_3,v_4)=(v_4,v_3,-v_2,-v_1),
\end{align*}
where $ (v_1,v_2,v_3,v_4)$ is a tangent vector to $\mathbb{R}^4$.

 It is simple to check that $(\mathbb{R}^4,J_i , \langle .,. \rangle)$ for $i=1,2,3$ are kahler manifolds. Moreover, suppose $N(p)=(p_1,p_2,p_3,p_4)$ for $p=(p_1,p_2,p_3,p_4) \in \mathbb{R}^4$  is the position vector field on $\mathbb{R}^4$. Let $W_1(p)=J_1N(p), W_2(p)=J_2N(p)$ and $W_3(p)=J_3N(p)$ be tangent vector fields to $S^3$.

We shall show that $W_1:(S^3,g) \longrightarrow (S(S^3),i^*g_{\delta ,0})$ is a  harmonic unit vector field, that is, it satisfies the corollary \ref{t1}.

Using the Gauss formula gives us the Levi-Civita connection $\nabla  ^S$ of $g$ as following
\begin{align}
 \nabla  ^S _Z W=\nabla ^E _ZW +g(Z,W)N.\label{46}
\end{align}
If $V$ is a vector field on $S^3$ then from (\ref{46}) and from the fact that $J_1$ is parallel with respect to $\nabla ^E$, one can get
\begin{align}
\nabla  ^S _V W_1&=\nabla ^E _VW_1 +g(V,W_1)N=\nabla ^E _VJ_1N +g(V,W_1)N
\nonumber \\
&=J_1V+g(V,W_1)N.\label{47}
\end{align}
By the definition, the Laplacian of $W_1$ is 
\begin{align*}
- \Delta _g W_1 &= \nabla ^S _{W_1} \nabla ^S _{W_1} W_1 + \nabla ^S _{W_2} \nabla ^S _{W_2} W_1 + \nabla ^S _{W_3} \nabla ^S _{W_3} W_1
 \\
& - \sum _{i=1} ^3 \nabla ^S _{\nabla ^S _{W_i}W_i} W_1 .\label{48}
\end{align*}
Using (\ref{46}) and (\ref{47}) and the fact that ${\nabla ^S _{W_i}W_i}=0$ (because the integral curves of $W_i$ for all $i=1,2,3$ are geodesics of $S^3$) give us
\begin{align*}
- \Delta _g W_1 &= J_1( \nabla ^E _{W_2} W_2 +\nabla ^E _{W_3} W_3).
\end{align*}
With an straight forward calculation, one can show that $\nabla ^E_{W_i} W_i=-N$ for $i=2,3$. So, we have
\begin{align*}
- \Delta _g W_1=-2J_1N=-2W_1.
\end{align*}
 On the other hand from (\ref{47}) we have 
\begin{align*}
||\nabla ^SW_1||^2=\sum _{i=1}^3g(\nabla ^S_{W_i}W_1,\nabla ^S_{W_i}W_1)=2.
\end{align*}
 Therefore, $W_1$ satisfies the equation $\Delta _g W_1=||\nabla W_1||^2W_1$.
\end{example}
\textbf{Acknowledgment}

The authors would like to thank Prof. R. Bryant for his valuable   answers to their questions in Mathoverflow.

\end{document}